\documentclass{article}
\usepackage[english]{babel}
\usepackage[letterpaper,top=2cm,bottom=2cm,left=3cm,right=3cm,marginparwidth=1.75cm]{geometry}
\usepackage{amsmath}
\usepackage{amssymb}
\usepackage{amsthm}
\usepackage{booktabs}
\usepackage{graphicx}
\usepackage[colorlinks=true,allcolors=blue]{hyperref}
\usepackage{tikz}
\usepackage{subcaption}
\usetikzlibrary{arrows.meta,positioning}

\tikzset{
 >=Latex,
 semithick,
 state/.style={circle,draw,minimum size=20pt,inner sep=0pt,font=\small},
 elab/.style={font=\footnotesize,fill=white,inner sep=1.5pt,align=center}
}
\theoremstyle{definition}
\newtheorem{definition}{Definition}[section]
\newtheorem{example}[definition]{Example}
\newtheorem{remark}[definition]{Remark}

\theoremstyle{plain}
\newtheorem{lemma}[definition]{Lemma}
\newtheorem{proposition}[definition]{Proposition}
\newtheorem{theorem}[definition]{Theorem}
\newtheorem{corollary}[definition]{Corollary}

\newcommand{\X}{\mathcal{X}}
\newcommand{\Sset}{\mathcal{S}}
\newcommand{\eps}{\epsilon}
\newcommand{\one}{\mathbf{1}}

\title{Identifiability of finite-state $\eps$-machines from finite summaries}
\author{Paweł Wieczyński\thanks{Ph.D. student, Institute of Computer Science, Polish Academy of Sciences, Warsaw, Poland. Email: \mbox{pawel.wieczynski@ipipan.waw.pl}.}}
\date{}

\begin{document}
\maketitle

\begin{abstract}
We study how much observable finite-dimensional information is sufficient to identify a finite-state generator $\eps$-machine. The observable summaries are labeled tables of conditional probabilities given finite histories and contiguous block laws. For stationary ergodic finite-alphabet processes, we first prove that a table with past horizon $L$ and future horizon $R$ is equivalent to the block law of length $L+R$. We then recall the classical linear-realization bound. Two hidden Markov presentations with $r$ and $s$ states agree on the entire process law if they agree on words through length $r+s-1$. Our main result gives a structural synchronization bound adapted to exact finite-state generator $\eps$-machines. If their target synchronization radii are $a$ and $b$, and their maximum predictive separation radius is $d$, then equality of one block law of length $a+b+d+1$ implies isomorphism. The proof aligns the states of both machines with words that synchronize the two generators simultaneously. For the full-support binary context subfamily of order $m$ with pairwise distinct parameters, the structural synchronization bound is $2m+2$, whereas the state-count and rank-dimension forms of the classical linear-realization bound are $2^{m+1}-1$. A construction shows that a horizon of order $m$ is necessary. When the order-$m$ context structure is known in advance, the sharp horizon is $m+1$ symbols. Thus the raw state-count and rank-dimension formulas can be exponentially pessimistic on this subfamily. The comparison is between general sufficient bounds, not with an optimal realization horizon.
\end{abstract}



\section{Introduction}\label{sec:introduction}

Observed stochastic systems often admit finite hidden representations even when the observed process is not Markov. An $\eps$-machine represents a process through predictive states obtained by identifying pasts that induce the same conditional law of the future. In the finite generator setting used here, these states form a strongly connected, unifilar, edge-emitting HMM with probabilistically distinct states. We ask which finite portion of the observable law already determines this presentation up to isomorphism. Without identifiability, distinct hidden mechanisms remain observationally equivalent and no estimator can consistently select the intended mechanism. The question addressed here is at the population level and is distinct from estimation from finite data.

\paragraph{Hidden Markov models and identifiability.}
Hidden Markov models (HMM) have been studied through functions of finite Markov chains, statistical inference, and information theory since the early work of Blackwell and Koopmans, Gilbert, and Baum and Petrie \cite{Blackwell_Koopmans_1957,Gilbert_1959,Baum_Petrie_1966}. See also the survey of Ephraim and Merhav \cite{Ephraim_Merhav_2002}. For the deterministic state-output subclass, Ito, Amari, and Kobayashi gave a necessary and sufficient linear-algebraic equivalence criterion and studied the resulting minimum effective degree of freedom \cite{Ito_Amari_Kobayashi_1992}. For stationary finite-output HMMs, Allman, Matias, and Rhodes use Kruskal's theorem after grouping observations into three blocks that are independent conditional on the central hidden state. They obtain generic identifiability up to state relabeling from a sufficiently long finite block \cite[Theorem~6]{Allman_Matias_Rhodes_2009}. For a stationary HMM with a known number $K$ of hidden states, Alexandrovich, Holzmann, and Leister obtain identification up to state relabeling from $2K+1$ observations under an ergodic full-rank transition matrix and pairwise distinct state-dependent laws \cite[Theorem~1]{Alexandrovich_Holzmann_Leister_2016}. These results concern general hidden presentations and their parameters. The present paper instead studies a canonical predictive presentation with a deterministic labeled state update.

Gassiat, Cleynen, and Robin identify a state-emitting HMM up to state relabeling from the law of three consecutive observations \cite{Gassiat_Cleynen_Robin_2016}. Their theorem assumes that the number of hidden states is known, the transition matrix has full rank, and the state-dependent emission laws are linearly independent. It is a useful benchmark based on a short window, but it does not apply to our entire model class. Our generators are edge-emitting, and our two general finite-horizon bounds assume neither full rank nor linear independence of emission laws. The difference in emission convention alone does not exclude an alternative state-emitting representation. On an alphabet of size $q$, however, at most $q$ probability mass functions can be linearly independent. Thus a binary state-emitting HMM satisfying their emission assumption has at most two states. Proposition~\ref{prop:context-rank} shows that the full-support binary context subfamily of order $m$ with pairwise distinct parameters has Hankel rank $2^m$. Since the Hankel rank is at most the state count of every HMM representation, for $m\ge2$ these processes admit no alternative HMM representation with at most two states. Consequently, no state-emitting representation of these processes satisfies the emission assumption of the three-observation theorem. We therefore compare numerical horizons only when the results apply to the same process class.

\paragraph{Computational mechanics.}
Computational mechanics makes the predictive construction above formal \cite{Crutchfield_Young_1989,Shalizi_Crutchfield_2001}. Its equivalence classes are causal states, and their minimal unifilar presentation is the $\eps$-machine. Travers and Crutchfield proved the equivalence of the finite-state generator definition with the history construction from predictive equivalence classes \cite{Travers_Crutchfield_2025}. Their synchronization theory distinguishes exact machines, for which a finite output word determines the current state, from nonexact machines, for which state uncertainty still vanishes asymptotically \cite{Travers_Crutchfield_2011a,Travers_Crutchfield_2011b}. The topology and empirical reconstruction of predictive states have also been studied \cite{Loomis_Crutchfield_2023,Shalizi_Shalizi_Crutchfield_2002}. These predictive and synchronization properties connect finite-state generator $\eps$-machines to stochastic and deterministic automata.

\paragraph{Stochastic automata, realization, and synchronization.}
A finite hidden presentation gives a linear representation of word probabilities. The rank of the associated Hankel matrix is the minimal dimension of an unrestricted real linear realization. Carlyle and Paz developed the connection between finite rank and realization for stochastic finite automata \cite{Carlyle_Paz_1971}. If two probabilistic automata with $r$ and $s$ states are inequivalent, Paz's result gives a separating word of length at most $r+s-1$ \cite{Paz_1971}. Tzeng later gave a polynomial-time linear-algebraic equivalence algorithm \cite{Tzeng_1992}. Observable operator models and modern HMM partial realization theory connect string probabilities, finite Hankel blocks, spectral methods, and hidden realizations \cite{Jaeger_2000,Huang_Ge_Kakade_Dahleh_2016}. This literature supplies the classical linear-realization bound used in Section~\ref{sec:observable}.

Short realization horizons are also known under genericity assumptions. For stationary HMMs with at most $K$ hidden states and alphabet size $q\ge2$, Huang, Ge, Kakade, and Dahleh obtain minimal quasi-HMM and HMM realizations from block laws with horizon of order $1+\log_q K$ when the transition and observation parameters are in general position \cite{Huang_Ge_Kakade_Dahleh_2016}. Thus logarithmic dependence on the state count is already established in generic realization theory. Our structural bound instead applies to every generator in the stated exact unifilar class without a genericity assumption. A result holding almost everywhere in the full HMM parameter space does not automatically hold on a prescribed structured subfamily. In particular, its applicability to the context family in Section~\ref{sec:context} would require a separate argument. Our synchronization and Hankel calculations establish the stated horizons directly for every member of the subfamily with pairwise distinct parameters.

Classical conformance testing provides another relevant automata-theoretic comparison. Chow's automata-theoretic testing method, now commonly presented as the W-method, combines words that reach states, one-step extensions that check successors, and distinguishing words that identify states \cite{Chow_1978}. Kocsis and Rot recently generalized this architecture to automata in monoidal closed categories and derived complete test suites for, among other models, weighted automata over a field \cite{Kocsis_Rot_2025}. In the weighted setting, the distinguishing condition acts on the full linear state space, rather than only requiring pairwise separation of the finitely many generator states.

Discarding transition probabilities from a finite unifilar generator leaves a partial deterministic automaton on its positive-probability labeled transition graph. Jonoska calls such a deterministic presentation synchronizing when every target vertex has an admissible word such that every path with that label terminates at the target \cite{Jonoska_1996}. Thus the existence of a synchronizing word for every target state is an established support property. Cai and Frongillo make the partial-DFA correspondence explicit. Their exact synchronization convention requires a word to be defined at least somewhere and to map every state on which it is defined to one state \cite{Cai_Frongillo_2022}. This is the support-level convention used here. Bradshaw, Clow, and Stacho study the related problem of using one word to synchronize two complete deterministic automata to prescribed targets \cite{Bradshaw_Clow_Stacho_2026}. Their setting has controlled, total transitions. Our setting is autonomous and stochastic. Admissible words depend on the state and may differ between models before equality of finite-word supports is imposed.

\paragraph{Current contribution and relation to prior work.}
The paper first establishes two baselines. Under stationarity and ergodicity, a labeled conditional table with past horizon $L$ and future horizon $R$ contains exactly the same information as the block law of length $L+R$. We also state and prove the classical $r+s-1$ linear-realization bound for two finite hidden presentations with $r$ and $s$ states. This second result is included as a benchmark and is not a contribution of the paper.

The general architecture of the main proof is classical: as in the W-method, one reaches states, checks one-symbol successors, and uses finite continuations to distinguish states \cite{Chow_1978,Kocsis_Rot_2025}. Target-by-target synchronization is also classical \cite{Jonoska_1996,Cai_Frongillo_2022}, and simultaneous synchronization has precedent for complete deterministic automata \cite{Bradshaw_Clow_Stacho_2026}. To the author's knowledge, the new point here is the way these ingredients are obtained and combined under the stationary observable law. Equality of finite-word supports is used to construct, for every state of either generator, a positive-probability history that is admissible in both models and synchronizes both, even though their positive-probability labeled transition graphs may initially differ. Equality of the corresponding conditional rows then matches the truncated state morphs. The predictive separation radius only needs to distinguish the finitely many generator states. Unlike the weighted-automata instance of the generalized W-method, it need not distinguish the full linear state space. This simultaneous-synchronization step yields the pair-dependent structural horizon $a+b+d+1$.

The target synchronization radius is the maximum, over states, of the shortest word length needed to synchronize to each state. The predictive separation radius is the maximum, over distinct pairs of states, of the shortest word length needed to distinguish their future laws. Formal definitions appear in Section~\ref{sec:synchronization}. If two exact finite-state generator $\eps$-machines have target synchronization radii $a$ and $b$, and the larger of their predictive separation radii is $d$, then equality of one block law of length $a+b+d+1$ implies isomorphism. On classes with target synchronization and predictive separation radii bounded by $A$ and $D$, the fixed block law of length $2A+D+1$ gives uniform identifiability.

An explicit binary context family with full Hankel rank shows how our structural bound can differ from the classical formulas. On the full-support subfamily of order $m$ with pairwise distinct parameters, the structural bound is $2m+2$. The raw state-count and rank-dimension bounds are $2^{m+1}-1$. A construction shows that a horizon of order $m$ is necessary. If the order-$m$ context structure is known in advance, the sharp horizon is $m+1$ symbols. These are comparisons between general sufficient bounds, not with an optimal realization horizon.

The results identify finite structural witnesses of equality. They do not provide estimators or algorithms for recognizing synchronizing rows from an arbitrary numerical table. Sections~\ref{sec:setting} and~\ref{sec:synchronization} fix the generator setting and introduce the two structural radii. Sections~\ref{sec:observable} and~\ref{sec:main} give the two finite-horizon bounds. Section~\ref{sec:context} studies the binary context family.

\section{Finite-state generator setting}\label{sec:setting}

Let $\X$ be a nonempty finite alphabet, let $\X^m$ denote the set of words of length $m$, let $\X^*$ denote the set of finite words, and let $\lambda$ be the empty word. For a bi-infinite process $(X_t)_{t \in \mathbb Z}$, write $X_i^j = (X_i, \ldots, X_j)$. We use capitals for random variables and small letters for their realizations, $X=x$.

\begin{definition}[Finite-state generator $\eps$-machine]\label{def:generator}
A finite-state generator $\eps$-machine is a quadruple
\begin{align*}
  G=(\Sset, \pi, \delta, \{ e_\sigma \}_{\sigma \in \Sset})
\end{align*}
with the following properties.
\begin{enumerate}
  \item $\Sset = \{ \sigma_1, \dots , \sigma_k \}$ is a nonempty finite set of states and every $e_\sigma$ is a probability mass function on $\X$.
  \item A state-symbol pair $(\sigma,x)$ is admissible when $e_\sigma(x)>0$, and the successor $\delta(\sigma,x)$ is defined on every admissible pair. Thus the presentation is unifilar.
  \item The directed graph of positive-probability labeled transitions is strongly connected.
  \item The conditional future laws are pairwise distinct: for $\sigma \ne \tau$, there is a word $w\in\X^*$ such that
  $$
    P(X_1^{|w|}=w\mid S_0=\sigma)
    \ne
    P(X_1^{|w|}=w\mid S_0=\tau).
  $$
  \item If
 $$
    T_{\sigma \tau} = \sum_{\substack{x : \, e_\sigma(x)>0
    \\ \delta(\sigma,x) = \tau}} e_\sigma(x),
  $$
  then $\pi$ is the unique probability vector satisfying $\pi T=\pi$.
\end{enumerate}
\end{definition}

Strong connectivity makes $T$ irreducible, so the stationary distribution $\pi$ exists and is unique~\cite{Norris_1997}. For admissible $(\sigma,x)$, we use the edge-emitting convention
\begin{equation}\label{eq:time-convention}
  P(X_{t+1}=x,S_{t+1}=\delta(\sigma,x)\mid S_t=\sigma)
  =e_\sigma(x).
\end{equation}
Equivalently, $X_{t+1}$ labels the transition from $S_t$ to $S_{t+1}$. The stationary joint law is obtained by taking $S_0\sim\pi$ and applying \eqref{eq:time-convention} in both time directions through the stationary Markov law. Define the finite set of positive-probability labeled edges by
\begin{equation*}
  \mathcal{E}
  =\bigl\{(\sigma,x,\delta(\sigma,x)): e_\sigma(x)>0\bigr\}.
\end{equation*}
The labeled-edge process $Y_t=(S_{t-1},X_t,S_t)$ is a stationary irreducible Markov chain on $\mathcal{E}$. With $\ell(\sigma,x,\tau)=x$, the coordinatewise map $(y_t)_{t\in\mathbb{Z}}\mapsto(\ell(y_t))_{t\in\mathbb{Z}}$ is measurable and commutes with the shift. Since $X_t=\ell(Y_t)$, the observable process is a stationary ergodic factor of this chain. Aperiodicity is unnecessary for ergodicity. Using labeled edges retains the emission randomness, which need not be determined by the state sequence alone.

For each $x\in\X$, define the symbol-labeled transition matrix
\begin{equation}\label{eq:symbol-matrix}
  T^{(x)}_{\sigma\tau}
  =\begin{cases}
    e_\sigma(x), & \text{if }(\sigma,x)\text{ is admissible and }\delta(\sigma,x)=\tau,\\
    0, & \text{otherwise}.
  \end{cases}
\end{equation}
Then $T=\sum_xT^{(x)}$. If $w=x_1\cdots x_m \in \X^m$, put
\begin{align*}
  T^{(w)}=T^{(x_1)}\cdots T^{(x_m)},
  \qquad T^{(\lambda)}=I,
\end{align*}
and define the state morph and word update by
\begin{align}
  p_\sigma(w)
  &=P(X_1^m=w\mid S_0=\sigma)
   =\bigl(T^{(w)}\one\bigr)_\sigma,\label{eq:morph}\\
  \delta(\sigma,w)
  &=\delta(\delta(\cdots\delta(\sigma,x_1),x_2)\cdots,x_m),
\end{align}
whenever the labeled path is admissible. Set $p_\sigma(\lambda)=1$ and $\delta(\sigma,\lambda)=\sigma$.

Two generators $G=(\Sset,\pi,\delta,e)$ and $G'=(\Sset',\pi',\delta',e')$ are \emph{isomorphic}, written $G\cong G'$, if there is a bijection $f:\Sset\to\Sset'$ such that
\begin{align}
  \pi'(f(\sigma))&=\pi(\sigma),\label{eq:iso-pi}\\
  e'_{f(\sigma)}(x)&=e_\sigma(x),\label{eq:iso-emission}\\
  \delta'(f(\sigma),x)&=f(\delta(\sigma,x))\label{eq:iso-transition}
\end{align}
for every admissible $(\sigma,x)$. Isomorphism handles only the arbitrary labeling of states.

For a stationary process $P$, its \emph{history causal states} are the equivalence classes of infinite pasts having the same conditional distribution of the future. In the finite setting used here, the equivalence theorem for history and generator presentations states that a process generated by a finite-state generator $\eps$-machine has, modulo null histories, a finite history machine isomorphic to the generator. Conversely, a stationary ergodic process with a finite history machine admits the corresponding generator \cite{Travers_Crutchfield_2025}. Thus equality of the observable process laws implies isomorphism of their finite-state generator $\eps$-machines. Probabilistic distinctness gives minimality within the canonical unifilar predictive class. It does not assert minimality among arbitrary, possibly nonunifilar, hidden presentations.

\section{Synchronization and predictive separation}\label{sec:synchronization}

Under convention \eqref{eq:time-convention}, observing $X_1^m$ moves the state from $S_0$ to $S_m$. Equivalently, a past block $X_{-m+1}^0$ ends at $S_0$. We use both stationary versions below.

\begin{definition}[Synchronizing word and exactness]\label{def:synchronizing-word}
A positive-probability word $w\in\X^m$, $m\ge0$, \emph{synchronizes to} $\sigma$ if
\begin{align*}
  P(S_m=\sigma\mid X_1^m=w)=1.
\end{align*}
Conditioning on the empty word $\lambda$ is trivial, so $\lambda$ synchronizes if and only if the generator has one state. The generator is \emph{exact} if it has a finite positive-probability synchronizing word. Under the standing assumptions, this is equivalent to the observer determining the current state after a finite number of observations almost surely \cite{Travers_Crutchfield_2011a}.
\end{definition}

By stationarity, if $w$ synchronizes to $\sigma$, then the same word observed as $X_{-m+1}^0=w$ determines $S_0=\sigma$. The Markov property therefore gives the identity for a pure row
\begin{equation}\label{eq:pure-row}
  P(X_1^{|v|}=v\mid X_{-m+1}^0=w)=p_\sigma(v),
  \qquad v\in\X^*.
\end{equation}

\begin{lemma}[Word characterization]\label{lem:word-characterization}
Let $w\in\X^m$ have positive probability. Then $w$ synchronizes to $\sigma$ if and only if every state $\tau$ from which $w$ is admissible satisfies $\delta(\tau,w)=\sigma$.
\end{lemma}

\begin{proof}
Since $\pi_\tau>0$ for every state,
\begin{align*}
P(S_m=\sigma\mid X_1^m=w)
=
\frac{\sum_{\substack{\tau:\,p_\tau(w)>0 \\ \delta(\tau,w)=\sigma}}\pi_\tau p_\tau(w)}
{\sum_\tau\pi_\tau p_\tau(w)}.
\end{align*}
Every admissible starting state has a strictly positive term in the denominator. The ratio equals one exactly when all such states end at $\sigma$.
\end{proof}

Two closure properties follow. If $u$ synchronizes to $\sigma$ and $zu$ is a positive-probability word, then $zu$ also synchronizes to $\sigma$. This is the suffix property. If $z$ synchronizes to $\sigma$, $u$ is admissible from $\sigma$, and $zu$ is a positive-probability word, then $zu$ synchronizes to $\delta(\sigma,u)$. This is synchronization followed by a unifilar update.

\begin{definition}[Structural radii]\label{def:radii}
For $\sigma\in\Sset$, let
\begin{align*}
  \ell_{\mathrm{tar}}(\sigma)
  =\inf\{|w|:w\text{ synchronizes to }\sigma\},
  \qquad \inf\emptyset=\infty,
\end{align*}
and define the \emph{target synchronization radius}
\begin{align*}
  L_{\mathrm{tar}}(G)=\max_{\sigma\in\Sset}\ell_{\mathrm{tar}}(\sigma).
\end{align*}
For distinct states $\sigma,\tau$, let
\begin{align*}
  \ell_{\mathrm{sep}}(\sigma,\tau)
  =\min\{|v|:p_\sigma(v)\ne p_\tau(v)\},
\end{align*}
and define the \emph{predictive separation radius}
\begin{align*}
  L_{\mathrm{sep}}(G)
  =\max_{\sigma\ne\tau}\ell_{\mathrm{sep}}(\sigma,\tau).
\end{align*}
For a generator with one state, set $L_{\mathrm{tar}}(G)=L_{\mathrm{sep}}(G)=0$.
Finally, the standard \emph{synchronization order} \cite{James_Mahoney_Ellison_Crutchfield_2014} is
\begin{align*}
  L_{\mathrm{syn}}(G)
  =\inf\{m:\text{every positive-probability word in }\X^m
    \text{ synchronizes}\},
  \qquad \inf\emptyset=\infty.
\end{align*}
\end{definition}

The target synchronization radius differs from synchronization order. It asks only for one synchronizing word per state, not for every word at a given length to be synchronizing. If a synchronizing word ends at $\sigma_0$, strong connectivity supplies a positive-probability path from $\sigma_0$ to any prescribed $\sigma$. Appending that path gives a synchronizing word to $\sigma$. Thus exactness implies $L_{\mathrm{tar}}(G)<\infty$. Probabilistic distinctness and finiteness similarly imply $L_{\mathrm{sep}}(G)<\infty$. The latter agrees with the maximum pairwise distinguishing length used by Travers and Crutchfield \cite{Travers_Crutchfield_2011b}. If two states agree on all words of length $m$, marginalization makes them agree at every shorter length.

The two radii also have elementary bounds that relate them to alphabet and state cardinalities.

\begin{proposition}[Elementary bounds for the structural radii] \label{prop:radius-bounds}
If $G$ has $r$ states, then
\begin{align}
  L_{\mathrm{sep}}(G)\le r-1. \label{eq:separation-state-bound}
\end{align}
Moreover, if $q=|\X|$ and $A\in\mathbb N_0$, then
\begin{align}
  L_{\mathrm{tar}}(G)\le A
  \quad\Longrightarrow\quad
  |\Sset|\le\sum_{\ell=0}^{A}q^\ell. \label{eq:target-count-bound}
\end{align}
\end{proposition}

\begin{proof}
For $j\ge0$, let
\begin{align*}
  W_j
  =\operatorname{span}\bigl\{T^{(w)}\one:|w|\le j\bigr\}
  \subseteq\mathbb R^r.
\end{align*}
The spaces $W_j$ are increasing and $\dim W_0=1$. If $W_{j+1}=W_j$, then $T^{(x)}W_j\subseteq W_j$ for every $x\in\X$, so the sequence has stabilized permanently. Since the ambient dimension is $r$, it follows that $W_{r-1}=W_r$ and that $W_{r-1}$ contains $T^{(w)}\one$ for every $w\in\X^*$.

Suppose that two states $\sigma$ and $\tau$ agree on all word probabilities through length $r-1$. By \eqref{eq:morph}, the linear functional $z\mapsto z_\sigma-z_\tau$ then annihilates $W_{r-1}$. Hence it annihilates $T^{(w)}\one$ for every finite word $w$. The two full state morphs are equal, contrary to probabilistic distinctness. Thus every distinct pair is separated by a word of length at most $r-1$, proving \eqref{eq:separation-state-bound}. The one-state convention gives the same conclusion when $r=1$.

For \eqref{eq:target-count-bound}, choose for each state one word of length at most $A$ that synchronizes to it. These chosen words are distinct, since one word cannot synchronize to two different states. There are $\sum_{\ell=0}^{A}q^\ell$ words of length at most $A$, which proves the claim.
\end{proof}

Thus a bound on $L_{\mathrm{tar}}$ already bounds the number of states, although for $q>1$ the resulting state-count bound grows exponentially in the target synchronization radius. This explains how a horizon expressed directly through the structural radii can scale differently from the raw state-count formula.

All finite-state generator $\eps$-machines are asymptotically synchronizable in the sense that the observer's posterior state entropy tends to zero, with exponential mean bounds under the standing assumptions~\cite{Travers_Crutchfield_2011a,Travers_Crutchfield_2011b}. Exact machines possess a finite positive-probability synchronizing word. Nonexact machines need not possess any pure row conditioned on a finite history. The main structural synchronization theorem in this paper concerns the exact class.

\begin{figure}[!htbp]
\centering
\begin{subfigure}[t]{0.47\textwidth}
\centering
\begin{tikzpicture}
\node[state] (A) at (0,0) {$A$};
\node[state] (B) at (3.6,0) {$B$};
\draw[->] (A) edge[loop left,min distance=9mm] node[elab,xshift=-2mm] {$\frac12\mid 0$} (A);
\draw[->] (A) edge[bend left=18] node[elab,above=1pt] {$\frac12\mid 1$} (B);
\draw[->] (B) edge[bend left=18] node[elab,below=1pt] {$1\mid 1$} (A);
\end{tikzpicture}
\caption{Even Process}
\label{fig:sync-even}
\end{subfigure}
\hfill
\begin{subfigure}[t]{0.47\textwidth}
\centering
\begin{tikzpicture}
\node[state] (A) at (0,0) {$A$};
\node[state] (B) at (3.6,0) {$B$};
\draw[->] (A) edge[bend left=18] node[elab,above=1pt] {$p\mid 1,\ 1-p\mid 0$} (B);
\draw[->] (B) edge[bend left=18] node[elab,below=1pt] {$q\mid 1,\ 1-q\mid 0$} (A);
\end{tikzpicture}
\caption{Alternating Biased Coin Process}
\label{fig:sync-alt}
\end{subfigure}

\vspace{1.2em}

\begin{subfigure}[t]{0.64\textwidth}
\centering
\begin{tikzpicture}
\node[state] (A) at (0,2.8) {$A$};
\node[state] (B) at (-3.7,-1.4) {$B$};
\node[state] (C) at (3.7,-1.4) {$C$};
\draw[->] (A) edge[loop above,min distance=9mm] node[elab,yshift=2pt] {$\frac12\mid 0$} (A);
\draw[->] (B) edge[loop left,min distance=9mm] node[elab,xshift=-2mm] {$\frac12\mid 1$} (B);
\draw[->] (C) edge[loop right,min distance=9mm] node[elab,xshift=2mm] {$\frac12\mid 2$} (C);
\draw[->] (A) edge[bend left=14] node[elab,pos=.54,xshift=-8pt] {$\frac14\mid 1$} (B);
\draw[->] (B) edge[bend left=14] node[elab,pos=.50,xshift=-8pt] {$\frac14\mid 0$} (A);
\draw[->] (A) edge[bend left=14] node[elab,pos=.54,xshift=8pt] {$\frac14\mid 2$} (C);
\draw[->] (C) edge[bend left=14] node[elab,pos=.50,xshift=8pt] {$\frac14\mid 0$} (A);
\draw[->] (B) edge[bend left=10] node[elab,above=2pt] {$\frac14\mid 2$} (C);
\draw[->] (C) edge[bend left=10] node[elab,above=2pt] {$\frac14\mid 1$} (B);
\end{tikzpicture}
\caption{Ternary Reset Process}
\label{fig:sync-ternary}
\end{subfigure}

\vspace{1.2em}

\begin{subfigure}[t]{0.64\textwidth}
\centering
\begin{tikzpicture}
\node[state] (A) at (0,2.8) {$A$};
\node[state] (B) at (-3.7,-1.4) {$B$};
\node[state] (C) at (3.7,-1.4) {$C$};
\draw[->] (A) edge[loop above,min distance=11mm] node[elab,yshift=3pt] {$\frac18\mid 0,\ \frac58\mid\star$} (A);
\draw[->] (B) edge[loop left,min distance=9mm] node[elab,xshift=-2mm] {$\frac18\mid 1$} (B);
\draw[->] (C) edge[loop right,min distance=11mm] node[elab,xshift=2mm] {$\frac14\mid 2,\ \frac14\mid\star$} (C);
\draw[->] (A) edge[bend left=14] node[elab,pos=.54,xshift=-8pt] {$\frac18\mid 1$} (B);
\draw[->] (B) edge[bend left=14] node[elab,pos=.50,xshift=-8pt] {$\frac18\mid 0$} (A);
\draw[->] (A) edge[bend left=14] node[elab,pos=.54,xshift=8pt] {$\frac18\mid 2$} (C);
\draw[->] (C) edge[bend left=14] node[elab,pos=.50,xshift=8pt] {$\frac14\mid 0$} (A);
\draw[->] (B) edge[bend left=10] node[elab,above=2pt] {$\frac18\mid 2,\ \frac58\mid\star$} (C);
\draw[->] (C) edge[bend left=10] node[elab,above=2pt] {$\frac14\mid 1$} (B);
\end{tikzpicture}
\caption{Delayed Separation Reset Process}
\label{fig:sync-delayed}
\end{subfigure}
\caption{The edge label $p\mid x$ means that symbol $x$ is emitted with probability $p$ on that edge.}
\label{fig:sync-examples}
\end{figure}
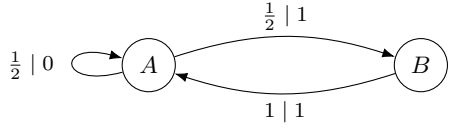
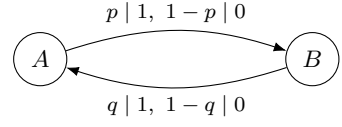
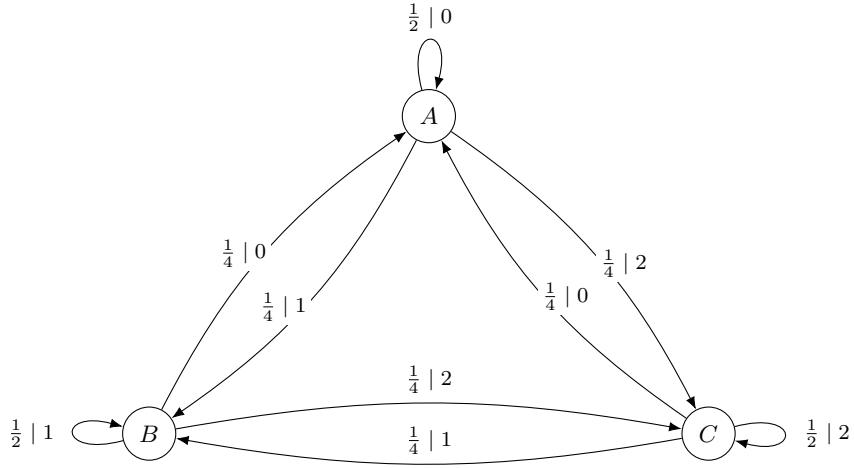
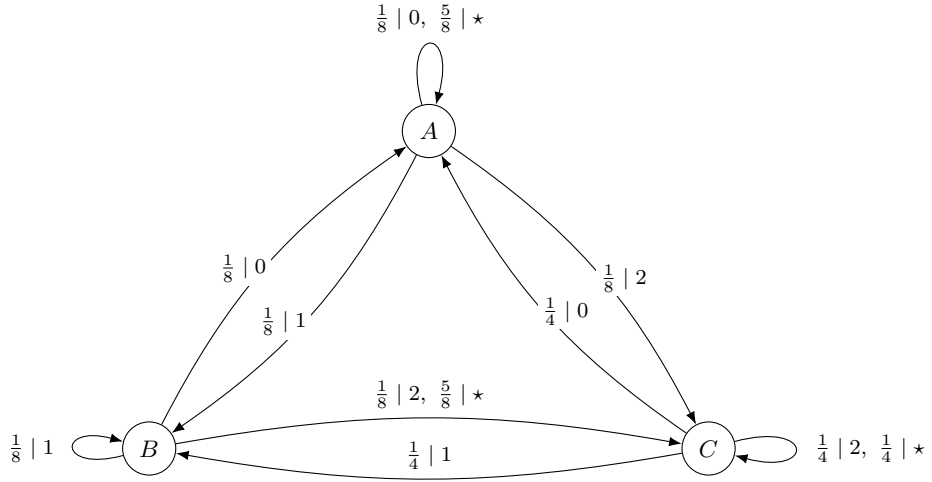

\begin{example}[Four synchronization regimes]\label{ex:four-machines}
The machines in Figure~\ref{fig:sync-examples} separate the synchronization and predictive-separation notions.
\begin{enumerate}
  \item In the \emph{Even Process}, $0$ synchronizes to $A$ and $01$ to $B$, while $0$ separates the states. Hence $L_{\mathrm{tar}}=2$ and $L_{\mathrm{sep}}=1$. $1^m$ is nonsynchronizing for every $m\ge1$, so $L_{\mathrm{syn}}=\infty$.
  \item In the \emph{Alternating Biased Coin Process}, assume $0<p,q<1$ and $p\ne q$. Both symbols are possible from both states and the phase always alternates, so $L_{\mathrm{tar}}=L_{\mathrm{syn}}=\infty$, whereas $L_{\mathrm{sep}}=1$.
  \item In the \emph{Ternary Reset Process}, symbols $0,1,2$ reset to $A,B,C$, respectively, and the one-symbol distributions are distinct. Thus $L_{\mathrm{tar}}=L_{\mathrm{syn}}=L_{\mathrm{sep}}=1$.
  \item In the \emph{Delayed Separation Reset Process}, $0,1,2$ again reset to $A,B,C$, so $L_{\mathrm{tar}}=1$. The word $\star^m$ remains nonsynchronizing for every $m$, and hence $L_{\mathrm{syn}}=\infty$. States $A$ and $B$ have the same distribution of the next symbol, but $\star0$ separates them. Consequently, $L_{\mathrm{sep}}=2$.
\end{enumerate}
\end{example}

\section{Observable summaries and finite equivalence}\label{sec:observable}

We now distinguish three claims that are sometimes conflated. Equality of process laws is an observational statement. Isomorphism of canonical generators follows from equality of process laws. Identification from a fixed summary is the stronger assertion that one prescribed finite observable already forces these conclusions throughout a class.

\begin{definition}[Uniform identifiability and pairwise finite distinguishability]\label{def:identifiability}
Let $\Theta$ be a class of stationary processes with finite-state generator $\eps$-machines, and let
\begin{align*}
  \{\Psi_h:\Theta\to\mathcal{Z}_h\}_{h\in\mathcal{H}}
\end{align*}
be a family of observable summaries indexed by finite horizons. The set $\mathcal{H}$ depends on the summary family. For block laws, $\mathcal{H}=\mathbb{N}_0=\{0,1,\ldots\}$ and $\Psi_n=P_n$. For conditional tables, $\mathcal{H}=\mathbb{N}^2$, where $\mathbb{N}=\{1,2,\ldots\}$. In this case $h=(L,R)$ and $\Psi_h=C_{L,R}$.

The family gives \emph{uniform identifiability} on $\Theta$ if one horizon can be chosen before the pair of processes is quantified:
\begin{equation}
  \exists h\in\mathcal{H}\ \ \forall P,Q\in\Theta:\qquad
  \Psi_h(P)=\Psi_h(Q)\quad\Longrightarrow\quad G_P\cong G_Q.
  \tag{U}\label{eq:uniform-identifiability}
\end{equation}
For a particular $h$ satisfying \eqref{eq:uniform-identifiability}, we also say that $\Psi_h$ identifies generators on $\Theta$. The family gives \emph{pairwise finite distinguishability} on $\Theta$ if the horizon may depend on the pair:
\begin{equation}
  \forall P,Q\in\Theta\ \ \exists h=h(P,Q)\in\mathcal{H}:\qquad
  \Psi_h(P)=\Psi_h(Q)\quad\Longrightarrow\quad G_P\cong G_Q.
  \tag{P}\label{eq:pairwise-distinguishability}
\end{equation}
Thus uniform identifiability has the quantifier order $\exists h\,\forall(P,Q)$, whereas pairwise finite distinguishability has the order $\forall(P,Q)\,\exists h$.
\end{definition}

For a stationary process $P$ and $n\ge0$, write
\begin{align*}
  P_n(w)=P(X_1^n=w),\qquad w\in\X^n,
\end{align*}
with $P_0(\lambda)=1$. For $L\ge1$, define the domain of positive-probability past words
\begin{align*}
  \mathcal{U}_L(P)
  =\left\{u\in\bigcup_{\ell=1}^L\X^\ell:
   P(X_{-|u|+1}^0=u)>0\right\}.
\end{align*}
For $u\in\mathcal{U}_L(P)$ and $v\in\{\lambda\}\cup\bigcup_{r=1}^R\X^r$, put
\begin{align*}
  P(v\mid u)
  =P(X_1^{|v|}=v\mid X_{-|u|+1}^0=u),
  \qquad P(\lambda\mid u)=1.
\end{align*}
The conditional summary
\begin{equation}\label{eq:conditional-table}
  C_{L,R}(P)
  =\bigl(P(v\mid u)\bigr)_{u\in\mathcal{U}_L(P),\,|v|\le R}
\end{equation}
is a \emph{labeled} table. Equality of two tables includes equality of their row domains of positive-probability words and coordinatewise equality at every pair $(u,v)$ indexed by words. The labels are part of the observable.

\begin{proposition}[Conditional tables and block laws]\label{prop:table-block}
Let $P$ and $Q$ be stationary ergodic processes on the same finite alphabet. For every $L,R\ge1$,
\begin{align*}
  C_{L,R}(P)=C_{L,R}(Q)
  \quad\Longleftrightarrow\quad
  P_{L+R}=Q_{L+R}.
\end{align*}
\end{proposition}

\begin{proof}
Suppose first that $P_{L+R}=Q_{L+R}$. Marginalization gives equality of all shorter contiguous block laws. By stationarity, for $|u|\le L$ and $|v|\le R$,
\begin{align*}
  P(v\mid u)=\frac{P_{|u|+|v|}(uv)}{P_{|u|}(u)}
\end{align*}
whenever the denominator is positive, and the same formula holds for $Q$. The positive-probability row domains and all table entries therefore agree. This direction uses stationarity but not ergodicity.

Conversely, suppose that the labeled tables agree. Let
\begin{align*}
  A_+=\{x\in\X:P(X_0=x)>0\}.
\end{align*}
The row domain for words of length one determines $A_+$. The columns indexed by single symbols determine the common stochastic matrix
\begin{align*}
  K(x,y)=P(X_1=y\mid X_0=x),\qquad x,y\in A_+.
\end{align*}
If $p(x)=P(X_0=x)$, stationarity gives $p=pK$.

We claim that $K$ is irreducible on $A_+$. Otherwise, the finite support graph has a nonempty proper closed communicating class $D\subsetneq A_+$. Since $p=pK$ and $p$ is positive on $A_+$,
\begin{align*}
  p(D)=p(D)+\sum_{x\notin D}p(x)K(x,D),
\end{align*}
so there is no positive flow into $D$ either. Therefore $\one\{X_t\in D\}=\one\{X_{t+1}\in D\}$ almost surely. The event that all coordinates lie in $D$ agrees modulo a null set with $\{X_0\in D\}$, is shift-invariant, and has probability strictly between zero and one. This contradicts ergodicity. Hence $K$ has a unique stationary distribution, and the common table determines the common single-symbol marginal $p$.

For a positive-probability word $u=x_1\cdots x_m$, $m\le L$, the chain rule gives
\begin{align*}
  P(u)=p(x_1)\prod_{j=1}^{m-1}P(x_{j+1}\mid x_1\cdots x_j).
\end{align*}
Every factor occurs in the common table. Words absent from the common row domain have probability zero under both laws. Thus $P(u)=Q(u)$ for $|u|\le L$. Finally, write $w=uv$ with $|u|=L$ and $|v|=R$. If $P(u)=Q(u)=0$, then both probabilities of $w$ vanish. Otherwise,
\begin{align*}
  P(w)=P(u)P(v\mid u)=Q(u)Q(v\mid u)=Q(w).
\end{align*}
Thus $P_{L+R}=Q_{L+R}$.
\end{proof}

Proposition~\ref{prop:table-block} is an equivalence between two summary types, not an identifiability result by itself. At matching horizons it transfers either uniform identifiability or pairwise finite distinguishability between conditional tables and block laws.

\begin{remark}\label{rem:table-assumptions}
Ergodicity is needed in the direction from tables to block laws. For $0<\theta<1$, consider the stationary nonergodic process law
\begin{align*}
  P_\theta=\theta\delta_{0^{\mathbb Z}}+(1-\theta)\delta_{1^{\mathbb Z}}.
\end{align*}
Thus $P_\theta$ generates the sequence $\dots000\dots$ with probability $\theta$ and the sequence $\dots111\dots$ with probability $1-\theta$. The processes $P_\theta$ have identical labeled conditional tables for all $\theta$ but different marginal weights. Stationarity aligns past and forward blocks and supplies $p=pK$. Finally, Proposition~\ref{prop:table-block} does not apply to an unlabeled row cloud, a multiset of rows, or a convex hull. Preservation of word labels and of the domain of positive-probability words is essential.
\end{remark}

For fixed $R\ge1$, consider the \emph{fixed-$R$ labeled-table sequence}
\begin{equation*}
  \bigl(C_{L,R}(P)\bigr)_{L\ge1}.
\end{equation*}
Each row retains its history-word label, and each table retains its positive-probability row domain. Consequently, tables with unbounded past horizons recover block laws of unbounded length. We record this elementary completeness fact to delimit the finite-horizon problem. Equality along an unbounded set of past horizons already determines any stationary ergodic finite-alphabet process, independently of a hidden presentation.

\begin{corollary}[Completeness of unbounded labeled tables]\label{cor:labeled-tables}
Under the assumptions of Proposition~\ref{prop:table-block}, if $C_{L,R}(P)=C_{L,R}(Q)$ for arbitrarily large $L$ and one fixed $R\ge1$, then $P=Q$. No assumption about hidden states, synchronization, or unifilarity is needed.
\end{corollary}

\begin{proof}
For any $n\ge1$, choose one of the common tables with $L+R\ge n$. Proposition~\ref{prop:table-block} gives $P_{L+R}=Q_{L+R}$, and marginalization gives $P_n=Q_n$. By stationarity, these contiguous block laws determine all finite-dimensional distributions, so $P=Q$.
\end{proof}

Corollary~\ref{cor:labeled-tables} uses an infinite sequence of finite summaries. It is therefore a completeness statement, not one of the two finite notions in Definition~\ref{def:identifiability}.

We compare two general types of sufficient finite-horizon bound. A \emph{classical linear-realization bound} has a horizon determined only by dimensions of linear representations. Applied directly to HMM presentations, it has a state-count form. Applied to smaller unrestricted real linear realizations, it has a rank-dimension form. It does not require unifilarity or synchronization. A \emph{structural synchronization bound} instead has a horizon determined by the target synchronization and predictive separation radii of exact finite-state generator $\eps$-machines. The process-adapted Hankel construction and reconstruction under known context structure in Section~\ref{sec:context} give sharper subclass-specific horizons. They are refinements, not additional general bound types.

We first state the classical linear-realization bound. An edge-emitting HMM with $r$ states has symbol matrices $T^{(x)}\in\mathbb R^{r\times r}$, an initial row distribution $\pi$, and
\begin{equation}\label{eq:linear-word-probability}
  P(w)=\pi T^{(w)}\one_r.
\end{equation}
No unifilarity or minimality is required. The \emph{process Hankel matrix} is the infinite matrix
\begin{equation}\label{eq:hankel}
  H_{u,v}=P(uv),\qquad u,v\in\X^*.
\end{equation}
Its rank is at most the dimension of every linear representation of the form \eqref{eq:linear-word-probability}. Finite-rank series and their unrestricted real linear realizations are classical \cite{Carlyle_Paz_1971,Jaeger_2000,Huang_Ge_Kakade_Dahleh_2016}. Entrywise nonnegativity is an additional positive realization constraint and is not implied by finite Hankel rank alone \cite{Vidyasagar_2011}.

\begin{proposition}[Classical linear-realization bound]\label{prop:finite-equivalence}
Let $P$ and $Q$ have edge-emitting finite HMM presentations with $r$ and $s$ states. If
\begin{align*}
  P(w)=Q(w)\qquad\text{for every }|w|\le r+s-1,
\end{align*}
then $P(w)=Q(w)$ for every finite word $w$.
\end{proposition}

\begin{proof}
Set $\nu=r+s$ and form the block-diagonal difference representation
\begin{align*}
  \gamma=(\pi_P,-\pi_Q),\qquad
  M^{(x)}=
  \begin{pmatrix}T_P^{(x)}&0\\0&T_Q^{(x)}\end{pmatrix},\qquad
  \eta=\begin{pmatrix}\one_r\\\one_s\end{pmatrix}.
\end{align*}
Then $P(w)-Q(w)=\gamma M^{(w)}\eta$. For $j\ge0$, let
\begin{align*}
  V_j=\operatorname{span}\{\gamma M^{(w)}:|w|\le j\}
  \subseteq\mathbb R^{1\times \nu}.
\end{align*}
The spaces increase with $j$. Once $V_{j+1}=V_j$, right multiplication by every $M^{(x)}$ leaves $V_j$ invariant, and the sequence has stabilized permanently. Since $\dim V_0=1$ and the ambient dimension is $\nu$, stabilization occurs by $j=\nu-1$. Agreement through length $\nu-1$ says that every spanning row of $V_{\nu-1}$ annihilates $\eta$. Every reachable row therefore annihilates $\eta$, so $\gamma M^{(w)}\eta=0$ for all $w$.
\end{proof}

The $r+s-1$ finite-witness bound is due to Paz \cite{Paz_1971}, while stabilization of a reachable span is the mechanism used in Tzeng's polynomial-time equivalence procedure \cite{Tzeng_1992}. It also makes clear that the statement is algebraic. Irreducibility, ergodicity, synchronization, and nonnegativity beyond the original HMM representations play no role in the proof.

\begin{corollary}[Pairwise finite distinguishability from block laws]\label{cor:universal-block}
If $P_n=Q_n$ for one $n\ge r+s-1$, then the two HMM presentations generate the same process law. If both presentations are finite-state generator $\eps$-machines, then $G_P\cong G_Q$.
\end{corollary}

\begin{proof}
For every word $w$ of length at most $r+s-1$, prefix marginalization gives
\begin{align*}
  P(w)=\sum_{z\in\X^{n-|w|}}P_n(wz)
    =\sum_{z\in\X^{n-|w|}}Q_n(wz)=Q(w).
\end{align*}
Proposition~\ref{prop:finite-equivalence} gives equality of all word probabilities. For stationary processes these determine the two-sided law. Canonical generator isomorphism then follows from the equivalence between history and generator presentations \cite{Travers_Crutchfield_2025}.
\end{proof}

On the class of finite-state generator $\eps$-machines with no uniform dimension restriction, Corollary~\ref{cor:universal-block} gives pairwise finite distinguishability in the sense of Definition~\ref{def:identifiability}: the horizon $r+s-1$ depends on the two presentation dimensions. If $P$ and $Q$ possess smaller linear realizations of dimensions $r^\prime\le r$ and $s^\prime\le s$, the same proof improves the classical linear-realization bound to $r^\prime+s^\prime-1$. Consequently, the raw number of hidden states is not always the right linear dimension.

\begin{corollary}[Uniform identifiability under bounded state count]\label{cor:state-count-class}
For $N\ge1$, let $\Theta_N$ be the class of stationary processes whose finite-state generator $\eps$-machines have at most $N$ states. Then the fixed block law $P_{2N-1}$ gives uniform identifiability on $\Theta_N$.
\end{corollary}

\begin{proof}
For $P,Q\in\Theta_N$, their state counts satisfy $r+s-1\le2N-1$. Corollary~\ref{cor:universal-block} applies after marginalization to the required shorter horizon.
\end{proof}

\begin{example}[Reconstruction under known reset structure]\label{ex:reset-table}
Suppose it is known in advance that the state set is $\{A,B,C\}$ and that symbols $0,1,2$ send every state from which they are admissible to $A,B,C$, respectively. Within this known reset structure, consider the Ternary Reset Process of Figure~\ref{fig:sync-ternary}, whose entire table $C_{1,1}$ is
\begin{equation*}
\begin{array}{cccc}
\toprule
u & P(0\mid u)&P(1\mid u)&P(2\mid u)\\
\midrule
0&1/2&1/4&1/4\\
1&1/4&1/2&1/4\\
2&1/4&1/4&1/2\\
\bottomrule
\end{array}.
\end{equation*}
The known reset rule makes each row the one-symbol distribution of the corresponding state. The three distinct rows therefore give the emissions, and after emitting $x$ the successor is the state corresponding to the row labeled $x$. The resulting state-transition matrix is doubly stochastic, so $\pi=(1/3,1/3,1/3)$. Thus the table recovers the parameters within the known reset structure. The table alone does not certify that its history labels are synchronizing words, so this is not identification of the structure from the table. Theorem~\ref{thm:pairwise} below does not assume that synchronizing rows can be recognized directly from an arbitrary table.
\end{example}

\section{A structural synchronization horizon}\label{sec:main}

Let $P$ and $Q$ be generated by exact finite-state generator $\eps$-machines $G_P$ and $G_Q$ over the same finite alphabet $\X$. Set
\begin{equation}\label{eq:abd}
  a=L_{\mathrm{tar}}(G_P),\qquad
  b=L_{\mathrm{tar}}(G_Q),\qquad
  d=\max\{L_{\mathrm{sep}}(G_P),L_{\mathrm{sep}}(G_Q)\}.
\end{equation}
The proof follows the classical steps of reaching states, checking successors, and distinguishing states \cite{Chow_1978,Kocsis_Rot_2025}. A history that synchronizes one generator may leave a posterior mixture of states in the other. The following lemma constructs a common history that reaches a definite state in both generators.

\begin{lemma}[Simultaneous synchronization]\label{lem:simultaneous}
Suppose that, for every $0\le \ell\le a+b$,
\begin{align*}
  \{w\in\X^\ell:P_\ell(w)>0\}
  =\{w\in\X^\ell:Q_\ell(w)>0\}.
\end{align*}
For every state $\sigma$ of $G_P$, there is a positive-probability word $w_\sigma$ of length at most $a+b$ that synchronizes $G_P$ to $\sigma$ and simultaneously synchronizes $G_Q$ to some state. The symmetric statement holds for every state of $G_Q$.
\end{lemma}

\begin{proof}
Choose $u_\sigma$ with $|u_\sigma|\le a$ that synchronizes $G_P$ to $\sigma$. Support equality gives $Q(u_\sigma)>0$, so some state $\tau$ of $G_Q$ admits it. Choose $z_\tau$ with $|z_\tau|\le b$ synchronizing $G_Q$ to $\tau$. The word
\begin{align*}
  w_\sigma=z_\tau u_\sigma
\end{align*}
has positive probability under $Q$. After $z_\tau$, the machine is at $\tau$, from which $u_\sigma$ is admissible. Since $|w_\sigma|\le a+b$, support equality gives it positive probability under $P$. The suffix property makes $w_\sigma$ synchronize $P$ to $\sigma$. Synchronization by $z_\tau$ followed by unifilarity makes it synchronize $Q$ to $\delta_Q(\tau,u_\sigma)$. Interchanging $P$ and $Q$ proves the symmetric statement.
\end{proof}

The lemma is a pair-specific technical step rather than an identifiability statement by itself. It yields the following theorem on pairwise finite distinguishability, whose horizons depend on the radii of $G_P$ and $G_Q$.

\begin{theorem}[Pairwise finite distinguishability from a conditional table]\label{thm:pairwise}
Let $L,R\ge1$ and suppose that $C_{L,R}(P)=C_{L,R}(Q)$. Then $G_P\cong G_Q$ under either sufficient allocation of the horizons:
\begin{align}
  L&\ge a+b+1,
  &R&\ge\max\{1,d\},\tag{A}\label{eq:condition-a}\\
  L&\ge a+b,
  &R&\ge d+1.\tag{B}\label{eq:condition-b}
\end{align}
\end{theorem}

\begin{proof}
\emph{State alignment.}
If $a=b=0$, both machines have one state. Choose a symbol $y$ in their common single-symbol support. A generator with one state produces an iid process, and for every $x\in\X$,
\begin{align*}
  e^P(x)=P(x\mid y)=Q(x\mid y)=e^Q(x).
\end{align*}
The machines are isomorphic. Assume henceforth that $a+b\ge1$.

Both \eqref{eq:condition-a} and \eqref{eq:condition-b} imply $L\ge a+b$ and $R\ge d$. Equality of the labeled row domains, together with stationarity, gives the finite-word support equality required by Lemma~\ref{lem:simultaneous}. For every state $\sigma$ of $G_P$, choose the common synchronizing word $w_\sigma$. It is nonempty in the present case and indexes a table row. If it ends at $\tau$ in $Q$, the identity for pure rows and equality of the table give
\begin{align*}
  p_\sigma^P(v)=P(v\mid w_\sigma)=Q(v\mid w_\sigma)=p_\tau^Q(v),
  \qquad |v|\le R.
\end{align*}
Applying the symmetric half of the lemma from every state of $G_Q$ shows that the two sets of state morphs truncated at horizon $R$ coincide. Since $R\ge d$, the truncated state morphs are pairwise distinct within both machines. They define a unique bijection
\begin{align*}
  f:\Sset_P\longrightarrow\Sset_Q,
  \qquad
  p_\sigma^P(v)=p_{f(\sigma)}^Q(v),\quad |v|\le R.
\end{align*}
By uniqueness of this matching, the state reached by $w_\sigma$ in $Q$ is $f(\sigma)$. At words of length one this gives
\begin{equation}\label{eq:matched-emissions}
  e_\sigma^P(x)=e_{f(\sigma)}^Q(x),
\end{equation}
so admissible state-symbol pairs correspond.

\emph{Transitions under (A).}
Fix an admissible $(\sigma,x)$. The word $w_\sigma x$ has positive probability under both laws, has length at most $a+b+1\le L$, and synchronizes the two machines to $\delta_P(\sigma,x)$ and $\delta_Q(f(\sigma),x)$, respectively. Its common row shows that these successor states have equal state morphs truncated at horizon $R$. By uniqueness of the matching of state morphs,
\begin{equation}\label{eq:successor-isomorphism}
  f(\delta_P(\sigma,x))=\delta_Q(f(\sigma),x).
\end{equation}

\emph{Transitions under (B).}
For $|v|\le R-1$, unifilarity and \eqref{eq:matched-emissions} give
\begin{align*}
  p_{\delta_P(\sigma,x)}^P(v)
  &=\frac{p_\sigma^P(xv)}{e_\sigma^P(x)}
   =\frac{p_{f(\sigma)}^Q(xv)}{e_{f(\sigma)}^Q(x)}\\
  &=p_{\delta_Q(f(\sigma),x)}^Q(v).
\end{align*}
The state $f(\delta_P(\sigma,x))$ has the same state morph under $Q$ as $\delta_P(\sigma,x)$ through length $R$. Hence it and $\delta_Q(f(\sigma),x)$ agree through length $R-1\ge L_{\mathrm{sep}}(G_Q)$ and must be the same state. Thus \eqref{eq:successor-isomorphism} holds under (B) as well.

Under either allocation, $f$ preserves all emissions and admissible labeled transitions. The state-transition matrices therefore agree up to the permutation induced by $f$. Their unique stationary distributions correspond, proving $G_P\cong G_Q$.
\end{proof}

Theorem~\ref{thm:pairwise} establishes pairwise finite distinguishability, not uniform identifiability on the unbounded exact class. Its horizons and common synchronizing words depend on both generators through their radii and Lemma~\ref{lem:simultaneous}. The proof does not provide a procedure for recognizing the synchronizing words from the table without further structure.

\begin{corollary}[Pairwise structural synchronization block bound]\label{cor:structural-block}
Under the assumptions of Theorem~\ref{thm:pairwise}, equality of a block law of length
\begin{equation}\label{eq:structural-horizon}
  n\ge a+b+d+1
\end{equation}
implies $G_P\cong G_Q$.
\end{corollary}

\begin{proof}
If $a=b=d=0$, both generators have one state and the block law of length one identifies their common emission distribution. Otherwise take $L=a+b$ and $R=d+1$ in condition (B). Proposition~\ref{prop:table-block} converts equality of the block law of length $L+R$ to equality of the conditional table. Equality at any longer horizon implies equality at $L+R$ by marginalization.
\end{proof}

Thus Corollary~\ref{cor:structural-block} is the block-law form of pairwise finite distinguishability: its structural synchronization horizon depends on the two generators.

\begin{corollary}[Uniform identifiability under bounded structural radii]\label{cor:fixed-class}
Fix $A,D\in\mathbb N_0$ and a finite alphabet $\X$, and let
\begin{align*}
\Theta_{\mathrm{ex}}(A,D)
=\bigl\{P:L_{\mathrm{tar}}(G_P)\le A,\ L_{\mathrm{sep}}(G_P)\le D,
\quad G_P\text{ is an exact finite-state generator $\eps$-machine}\bigr\}.
\end{align*}
Subject to $L,R\ge1$, each of
\begin{align*}
  C_{2A+1,\max\{1,D\}},
  \qquad
  C_{\max\{1,2A\},D+1}
\end{align*}
gives uniform identifiability on $\Theta_{\mathrm{ex}}(A,D)$. The block law $P_{2A+D+1}$ also gives uniform identifiability on this class.
\end{corollary}

\begin{proof}
For any $P,Q\in\Theta_{\mathrm{ex}}(A,D)$, the quantities in \eqref{eq:abd} satisfy $a,b\le A$ and $d\le D$. The two conditional-table claims follow from the corresponding allocations in Theorem~\ref{thm:pairwise}, and the block-law claim follows from Corollary~\ref{cor:structural-block} after marginalization.
\end{proof}

This is uniform identifiability in the sense of Definition~\ref{def:identifiability}: the summary is chosen before the processes are quantified. Theorem~\ref{thm:pairwise} and Corollary~\ref{cor:structural-block} instead give pairwise finite distinguishability with horizons adapted to a given pair.

\section{Binary context machines}\label{sec:context}

We compare the structural synchronization and classical linear-realization bounds on an explicit family with full Hankel rank. The \emph{Markov order} of a stationary process is the least $k\ge0$ for which the conditional law of the next symbol given the entire past depends only on the last $k$ symbols, almost surely. For $k=0$, this law is independent of the past. If no finite $k$ exists, the Markov order is infinite.

The binary \emph{de Bruijn graph of order $m$} has vertex set $\{0,1\}^m$ and an edge labeled $x$ from $u=u_1\cdots u_m$ to $u_2\cdots u_mx$. Thus its state records the most recent $m$ symbols.

Fix $m\ge1$ and let
\begin{equation}\label{eq:context-machine}
  \Sset_m=\{0,1\}^m,
  \qquad
  \delta(u,x)=u_2\cdots u_mx,
\end{equation}
where the right-hand side is $x$ when $m=1$. Choose parameters $p_u\in(0,1)$ and set
\begin{align*}
  e_u(1)=p_u,
  \qquad
  e_u(0)=1-p_u.
\end{align*}
This is the full-support binary context presentation of order $m$. The following proposition considers the subfamily with pairwise distinct parameters.

\begin{proposition}[Radii and full Hankel rank]\label{prop:context-rank}
If the $2^m$ numbers $(p_u)_{u\in\Sset_m}$ are pairwise distinct, then \eqref{eq:context-machine} is an exact finite-state generator $\eps$-machine and
\begin{align*}
  |\Sset_m|=2^m,
  \qquad
  L_{\mathrm{tar}}=m,
  \qquad
  L_{\mathrm{sep}}=1.
\end{align*}
Its standard synchronization order and its Markov order are also $m$, and its process Hankel matrix has full rank $2^m$.
\end{proposition}

\begin{proof}
Every labeled edge has positive probability, and the de Bruijn graph is strongly connected. For $0\le r\le m$, let $\operatorname{suffix}_r(u)$ denote the terminal subword of $u$ of length $r$, with $\operatorname{suffix}_0(u)=\lambda$. If $w$ has length $\ell\le m$, then
\begin{equation}\label{eq:context-update}
  \delta(u,w)=\operatorname{suffix}_{m-\ell}(u)w.
\end{equation}
Every $w\in\{0,1\}^m$ therefore sends all initial states to state $w$, so it synchronizes to $w$. If $\ell<m$, choose two initial contexts having different suffixes of length $m-\ell$. Both admit $w$, but \eqref{eq:context-update} gives different endpoints. No shorter word synchronizes. Hence $L_{\mathrm{tar}}=L_{\mathrm{syn}}=m$.

For every context $u$,
\begin{align*}
  P(X_1=1\mid S_0=u)=p_u.
\end{align*}
Pairwise distinctness separates every state pair in one step, so $L_{\mathrm{sep}}=1$ and the states are probabilistically distinct. Under the stationary law,
\begin{align*}
  S_t=X_{t-m+1}^t\qquad\text{almost surely},
\end{align*}
so the output has Markov order at most $m$. An order below $m$ would imply $p_{0z}=p_{1z}$ for every $z\in\{0,1\}^{m-1}$, contrary to pairwise distinctness. Thus the Markov order is exactly $m$.

It remains to determine the Hankel rank. Let $T_m$ be the $2^m\times2^m$ state-transition matrix. Group its rows as $(0z,1z)$ and columns as $(z0,z1)$ for $z\in\{0,1\}^{m-1}$. After independent row and column permutations, $T_m$ is block diagonal with blocks
\begin{align*}
  B_z=
  \begin{pmatrix}
    1-p_{0z}&p_{0z}\\
    1-p_{1z}&p_{1z}
  \end{pmatrix},
  \qquad
  \det B_z=p_{1z}-p_{0z}.
\end{align*}
Hence
\begin{align*}
  \det T_m=\pm\prod_{z\in\{0,1\}^{m-1}}(p_{1z}-p_{0z})\ne0.
\end{align*}

Restrict the process Hankel matrix \eqref{eq:hankel} to rows and columns indexed by $u,v\in\{0,1\}^m$. Observing $u$ synchronizes the state to $u$, and after a further $m$ symbols the terminal state equals the emitted word $v$. Consequently,
\begin{align*}
  H^{(m,m)}_{u,v}=P_m(u)(T_m^m)_{u,v},
\end{align*}
or
\begin{align*}
  H^{(m,m)}
  =\operatorname{diag}\bigl(P_m(u):u\in\{0,1\}^m\bigr)T_m^m.
\end{align*}
Every $P_m(u)$ is positive and $T_m$ is invertible, so this finite Hankel block is invertible. Therefore $\operatorname{rank}H\ge2^m$. The HMM representation with $2^m$ states gives the reverse inequality.
\end{proof}

The determinant calculation requires only $p_{0z}\ne p_{1z}$ for each $z$. Pairwise distinctness is the stronger assumption used to obtain $L_{\mathrm{sep}}=1$.

\begin{proposition}[A pair giving a lower bound at a finite horizon]\label{prop:context-lower}
For every $m\ge1$, there are two nonisomorphic machines satisfying all conclusions of Proposition~\ref{prop:context-rank} with the same block law of length $m$ and different block laws of length $m+1$. Consequently, $P_m$ does not give uniform identifiability on $\Theta_{\mathrm{ex}}(m,1)$.
\end{proposition}

\begin{proof}
Let $M=2^{m-1}$ and enumerate $z\in\{0,1\}^{m-1}$ by $j(z)\in\{0,\ldots,M-1\}$. More generally, choose pairwise distinct $\theta_z\in(0,1/2)$ and set
\begin{equation}\label{eq:symmetric-context-parameters}
  p_{0z}=\theta_z,
  \qquad
  p_{1z}=1-\theta_z.
\end{equation}
All $2^m$ parameters are pairwise distinct. For each $z$, the only nonzero entries in columns $z0$ and $z1$ come from rows $0z$ and $1z$, and
\begin{align*}
  (1-p_{0z})+(1-p_{1z})=1,
  \qquad
  p_{0z}+p_{1z}=1.
\end{align*}
The transition matrix is therefore doubly stochastic. Its stationary distribution is uniform, and the state after any $m$ outputs is the emitted block of length $m$. Hence
\begin{equation}\label{eq:uniform-m-block}
  P_m(u)=2^{-m},\qquad u\in\{0,1\}^m.
\end{equation}

Now construct $P$ and $Q$ using, respectively,
\begin{align*}
  \theta_z=\frac{j(z)+1}{4M},
  \qquad
  \theta'_z=\frac{j(z)+1}{6M}.
\end{align*}
Both parameter collections lie in $(0,1/2)$ and are pairwise distinct. Equation \eqref{eq:uniform-m-block} makes their block laws of length $m$ equal, whereas
\begin{align*}
  P_{m+1}(0z1)&=2^{-m}\frac{j(z)+1}{4M},\\
  Q_{m+1}(0z1)&=2^{-m}\frac{j(z)+1}{6M}.
\end{align*}
Thus the block laws of length $m+1$ differ. Since the two processes differ, their finite-state generator $\eps$-machines cannot be isomorphic.
\end{proof}

\begin{corollary}[Comparison of horizons]\label{cor:context-comparison}
For two machines from Proposition~\ref{prop:context-rank}, the structural synchronization bound in Corollary~\ref{cor:structural-block} is
\begin{align*}
  n_{\mathrm{struct}}=2m+2,
\end{align*}
whereas the state-count and rank-dimension forms of the classical linear-realization bound are
\begin{align*}
  n_{\mathrm{count}}=n_{\mathrm{rank}}=2^{m+1}-1.
\end{align*}
The structural synchronization bound is smaller for every $m\ge2$, and
\begin{align*}
  \frac{n_{\mathrm{count}}}{n_{\mathrm{struct}}}
  =\frac{2^{m+1}-1}{2m+2}
  =\Theta\!\left(\frac{2^m}{m}\right).
\end{align*}
Moreover, no contiguous block horizon at most $m$ gives uniform identifiability on the class $\Theta_{\mathrm{ex}}(m,1)$.
\end{corollary}

\begin{proof}
For both machines, $a=b=m$ and $d=1$, which gives $2m+2$. Each machine has $2^m$ states and Hankel rank $2^m$, so Proposition~\ref{prop:finite-equivalence} gives $2^{m+1}-1$ from either form of the classical linear-realization bound. Proposition~\ref{prop:context-lower} supplies two members of $\Theta_{\mathrm{ex}}(m,1)$ that agree through the block law of length $m$ and are not isomorphic.
\end{proof}

The weighted-automata W-method gives a further comparison on the family of Proposition~\ref{prop:context-rank}. Let $P$ and $Q$ be two members of this family for the same $m$. View their stationary presentations as real weighted automata, with initial row $\pi$, symbol matrices $T^{(x)}$, and terminal vector $\one$ in our convention. Take the word set
\begin{align*}
  B=\{\lambda\}\cup\{0,1\}^m.
\end{align*}
For the presentation of $P$ and each $u\in\{0,1\}^m$, the reachable row $\pi T^{(u)}$ has its only nonzero entry at state $u$, with value $P_m(u)>0$. These rows span the full state space. The same argument applies to $Q$, so $B$ is a state cover for both presentations in the linear sense of \cite[Proposition~5.4(i)]{Kocsis_Rot_2025}. In particular, it covers the competing model, as required by their fault domain at order zero. The matrix $(p_u(v))_{u,v\in\{0,1\}^m}=T_m^m$ is invertible by Proposition~\ref{prop:context-rank}. Hence the continuations in $B$ distinguish all vectors in the linear state space, making $B$ a characterization set in the sense of \cite[Proposition~5.4(ii)]{Kocsis_Rot_2025}. Full Hankel rank also ensures that the presentations are minimal as weighted automata.

Applying \cite[Theorem~5.7]{Kocsis_Rot_2025} at order zero therefore gives the complete test suite
\begin{align*}
  B\cdot\bigl(\{\lambda\}\cup\{0,1\}\bigr)\cdot B,
\end{align*}
where the dots denote concatenation of word sets. Its maximal word length is $2m+1$. Equality of $P_{2m+1}$ and $Q_{2m+1}$ implies agreement on every test by marginalization, and the theorem then gives equality of the entire process laws. This application uses the known family to verify the state cover condition for the competitor. It shows that a conformance test adapted to this family, like the realization below, retains linear order in $m$. Corollary~\ref{cor:context-comparison} compares the general state-count and rank-dimension formulas.

The exponential comparison must be read narrowly. It compares two general sufficient bound types on the full-support binary context subfamily with pairwise distinct parameters from Proposition~\ref{prop:context-rank}. It is not an exponential lower bound against all realization methods. Indeed, this family has an invertible Hankel basis block at word length $m$. We apply the standard linear realization from a fixed Hankel basis \cite{Jaeger_2000,Huang_Ge_Kakade_Dahleh_2016}. Put $U=V=\{0,1\}^m$ and define
\begin{align*}
  H_{U,V}&=(P(uv))_{u\in U,\,v\in V},\qquad
  H^{(x)}_{U,V}=(P(uxv))_{u\in U,\,v\in V},
  \quad x\in\{0,1\}.
\end{align*}
The matrix $H_{U,V}=H^{(m,m)}$ is invertible by Proposition~\ref{prop:context-rank}. Define
\begin{align}
  A_x&=H^{(x)}_{U,V}H_{U,V}^{-1},\label{eq:adapted-operator}\\
  \alpha&=(P(v))_{v\in V}H_{U,V}^{-1},
  \quad \beta =(P(u))_{u\in U}^{\mathsf T}.\label{eq:adapted-boundaries}
\end{align}
We verify the realization. Let $(\rho,(M_x)_{x\in\X},\omega)$ be a minimal $2^m$-dimensional linear representation of the process, with $M_u=M_{u_1}\cdots M_{u_{|u|}}$. Let $F_U$ have rows $\rho M_u$, $u\in U$, and let $G_V$ have columns $M_v\omega$, $v\in V$. Then
\begin{align*}
  H_{U,V}=F_U G_V,
  \qquad
  H^{(x)}_{U,V}=F_U M_x G_V.
\end{align*}
Invertibility of $H_{U,V}$ makes both $F_U$ and $G_V$ invertible. Consequently,
\begin{align*}
  A_x=F_U M_x F_U^{-1},\qquad
  \alpha=\rho F_U^{-1},\qquad
  \beta=F_U\omega.
\end{align*}
Therefore, for every $w=w_1\cdots w_k$,
\begin{align*}
  \alpha A_{w_1}\cdots A_{w_k}\beta
  =\rho M_{w_1}\cdots M_{w_k}\omega
  =P(w).
\end{align*}
This representation is minimal. The invertible $2^m\times2^m$ Hankel block gives a lower bound of $2^m$ on the full Hankel rank, and the context presentation supplies the matching upper bound. The quantities in \eqref{eq:adapted-operator}--\eqref{eq:adapted-boundaries} use block probabilities only through length $2m+1$, one symbol shorter than the structural synchronization bound. For a process in the family of Proposition~\ref{prop:context-rank}, they determine its entire law among competing processes whose Hankel rank is at most $2^m$. Indeed, equality of the block laws gives the same invertible basis block and the same realization formulas for both processes. This rank restriction is part of the identification claim. Invertibility of a finite Hankel block alone does not certify that the full process rank is no larger. This is a process-adapted refinement of linear realization rather than the general classical linear-realization formula. If the context structure of order $m$ is known in advance, the identifying horizon is smaller still:
\begin{equation}\label{eq:context-reconstruction}
  p_u=\frac{P_{m+1}(u1)}{P_m(u)}.
\end{equation}
Thus $P_{m+1}$ reconstructs every emission probability and the de Bruijn update is known. Proposition~\ref{prop:context-lower} shows that $m+1$ is sharp when the order-$m$ context structure is known. The conclusion is that the rank-dimension form of the classical linear-realization bound can be exponentially pessimistic, while the structural synchronization bound and a realization adapted to the process both retain logarithmic order in the Hankel rank.

\section{Discussion and future research}\label{sec:discussion}

The paper establishes two complementary general principles for finite horizons. The classical linear-realization bound gives the universal horizon $r+s-1$ for arbitrary finite hidden presentations. Exact finite-state generator $\eps$-machines also admit the structural synchronization bound $a+b+d+1$, where $a$ and $b$ are target synchronization radii and $d$ measures predictive separation. Neither bound dominates the other in all models. On the full-support binary context subfamily with pairwise distinct parameters, the structural formula is exponentially shorter than the state-count and rank-dimension forms of the classical bound. An invertible Hankel basis block or prior knowledge of the order-$m$ context structure gives a sharper horizon than either general formula. Thus the example compares sufficient formulas rather than establishing an exponential advantage over optimal realization.

The overall reach-successor-distinguish architecture of the structural proof has a classical precedent in the W-method. The simultaneous synchronization lemma is the step specific to the present stationary setting. A word that synchronizes $P$ may be only a posterior mixture of states under $Q$. Choose a state of $Q$ from which the synchronizing word for $P$ is admissible. Prepending a synchronizing word for $Q$ that reaches this state produces a common word that synchronizes both models. The resulting symmetric term $a+b$ is the price of state alignment in this proof. Appending one symbol then recovers successors. Alternatively, one extra future symbol permits normalized shifting of the matched state morphs.

Target-by-target synchronization is already present in Jonoska's synchronizing deterministic presentations of sofic shifts \cite{Jonoska_1996}. Cai and Frongillo explicitly formulate deterministic presentations as partial DFAs and identify the relevant convention with exact synchronization \cite{Cai_Frongillo_2022}. The prescribed-target common-synchronization problem for two complete deterministic automata is explicitly studied by Bradshaw, Clow, and Stacho \cite{Bradshaw_Clow_Stacho_2026}. The concatenation in Lemma~\ref{lem:simultaneous} is therefore a support-level automata argument, not a new probabilistic principle. Its role here is to exploit equality of finite-word supports even though the two autonomous generators may have different partial transition maps and statewise admissibility domains. Matching truncated state morphs then yields the explicit $a+b+d+1$ finite identification horizon. To the author's knowledge, this combination and bound have not previously been stated for exact finite-state generator $\eps$-machines under their stationary laws.

\paragraph{Beyond exact synchronization.}
Every finite-state generator $\eps$-machine is asymptotically synchronizable \cite{Travers_Crutchfield_2011b}. The classical linear-realization bound gives pairwise finite distinguishability for every pair of finite-state generator $\eps$-machines. The open problem is therefore not existence of some finite identifying block for a fixed pair. Rather, it is whether asymptotic synchronization yields a useful structural synchronization horizon, stability modulus, or recognizable population construction that improves on either the state-count form or the rank-dimension form of the classical linear-realization bound. An unbounded labeled-table sequence is not a nontrivial answer. Corollary~\ref{cor:labeled-tables} makes it complete for all stationary ergodic finite-alphabet processes. A genuinely different problem would use an unlabeled row cloud or another lossy observable. In that case, state recognition and successor recovery would have to be proved anew.

\paragraph{Identifiability, stability, and estimation.}
Identifiability alone does not control sampling error. A stability theory should relate perturbations of observable summaries to distances between generator isomorphism classes, with bounds reflecting probabilities of synchronizing words, separation of truncated state morphs, and sensitivity of the stationary distribution. Such bounds could support consistent estimation and finite-sample guarantees. Convergence of predictive states \cite{Loomis_Crutchfield_2023} and CSSR consistency arguments in exact regimes \cite{Shalizi_Shalizi_Crutchfield_2002} provide starting points.

\paragraph{Infinite predictive state spaces.}
Countably infinite generators require replacements for the finite synchronization and separation radii based on local conditions or stationary mass. Uncountable predictive state spaces additionally require a measurable predictive transition kernel and a topology compatible with prediction. Extending identifiability, stability, and estimation to these settings remains outside the finite-state theory proved here.

\section*{Statements and Declarations}

\paragraph{Funding.}
This research did not receive any specific grant from funding agencies in the public, commercial, or not-for-profit sectors.

\paragraph{Declaration of competing interest.}
The author declares that there are no known competing financial interests or personal relationships that could have appeared to influence the work reported in this paper.

\paragraph{Data availability.}
No data were used or generated for the research described in this article.

\paragraph{Declaration of generative AI and AI-assisted technologies in the manuscript preparation process.}
Generative AI tools, namely OpenAI GPT-5.6 Sol via Codex, were used during manuscript preparation for language editing, literature organization, notation consistency checks, and critical scrutiny of mathematical arguments. The author reviewed and edited all AI-assisted output and takes full responsibility for the content of the manuscript.

\bibliographystyle{plain}
\bibliography{references}

\end{document}